\documentclass[a4paper, 12pt]{article}
\usepackage[margin=25mm]{geometry}

\usepackage{amsmath}
\usepackage{amsfonts} 
\usepackage{amssymb}
\usepackage{amsthm} 
 
\usepackage{graphicx}     

\usepackage{subfigure} 

\theoremstyle{plain}
\newtheorem{theorem}{Theorem}[section]

\newtheorem{cor}[theorem]{Corollary}

\newtheorem{lem}[theorem]{Lemma}

\newtheorem{prop}[theorem]{Proposition}

\theoremstyle{definition}

\newtheorem{rem}[theorem]{Remark}
\providecommand{\keywords}[1]
{
  \small	
  \textbf{\textit{Keywords.}} #1
}

\providecommand{\msc}[1]
{
  \small	
  \textbf{\textit{2020 Mathematics Subject Classification.}} #1
}

\title{
Boundary layer profiles of positive solutions for logistic equations with sublinear nonlinearity on the boundary
}

\author{Kenichiro Umezu  \\
        \small Department of Mathematics, Ibaraki University, Mito 310--8512, Japan \\
        \small E-mail: \texttt{kenichiro.umezu.math@vc.ibaraki.ac.jp}  \\
}

\numberwithin{equation}{section}
\numberwithin{theorem}{section}

\allowdisplaybreaks[4] 

\date{} 
\begin{document} 
\maketitle

\begin{abstract}
In this paper, we consider the logistic elliptic equation $-\Delta u = u- u^{p}$ in a smooth bounded domain $\Omega \subset \mathbb{R}^{N}$, $N\geq2$, equipped with the sublinear Neumann boundary condition $\frac{\partial u}{\partial \nu} = \mu u^{q}$ on $\partial \Omega$, where $0<q<1<p$, and $\mu\geq0$ is a parameter. 
With sub-- and supersolutions and a comparison principle for the equation, we analyze the asymptotic profile of a unique positive solution for the equation as $\mu \to \infty$. 
\end{abstract} \hspace{10pt}

\keywords{Semilinear elliptic equation, logistic nonlinearity, sublinear and non--Lipschitz boundary condition, sub-- and supersolutions, comparison principle, 
positive solution curve, boundary layer profile.}

\msc{35B09, 35B30, 35B40, 35B51, 35J25, 35J66.}   



\section{Introduction} 
\label{sec:intro} 

This paper is devoted to the study of the following logistic elliptic equation, equipped with a nonlinear boundary condition with non--Lipschitz nonlinearity at zero and sublinearity at infinity. 
\begin{align} \label{p}
\begin{cases}
-\Delta u = u-u^{p} & \mbox{ in } \Omega,  \\
\frac{\partial u}{\partial \nu}  = \mu u^q & \mbox{ on } \partial\Omega. 
\end{cases} 
\end{align}
Here, $\Omega\subset \mathbb{R}^N$, $N\geq2$, is a bounded domain with a smooth boundary $\partial\Omega$, 
$\Delta = \sum_{i=1}^N \frac{\partial^2}{\partial x_i^2}$ is the usual Laplacian in $\mathbb{R}^N$, $0<q<1<p$, $\mu \geq 0$ is a parameter, and $\nu$ is the unit outer normal to $\partial\Omega$. 

We call a nonnegative function $u\in H^1(\Omega)\cap C(\overline{\Omega})$ a \textit{nonnegative solution} of \eqref{p} if $u$ satisfies 
\begin{align} \label{sol:def}
\int_{\Omega} \biggl( \nabla u \nabla \varphi -u\varphi + u^p \varphi \biggr)dx - \mu \int_{\partial\Omega} u^{q} \varphi \, d\sigma =0, \quad \varphi \in H^1(\Omega). 
\end{align} 
We see that $u=0$ is a nonnegative solution of \eqref{p} for every $\mu\geq 0$, called a {\it trivial solution}. 
We know (\cite[Theorem 2.2]{Ro2005}) that if we additionally assume $p<\frac{N+2}{N-2}$ for $N>2$, then $0\leq u \in H^{1}(\Omega)$ satisfying \eqref{sol:def} belongs to $W^{1,r}(\Omega)$ for $r>N$, and consequently, to $C^{\theta}(\overline{\Omega})$ for $0<\theta<1$. Thus, $u$ is a nonnegative solution of \eqref{p}. 
A nonnegative solution $u$ of \eqref{p} that is nontrivial implies $u\in C^{2+\theta}(\Omega)$ for $0<\theta<1$ by elliptic regularity (\cite{GT83}), and hence, $u>0$ in $\Omega$ by the strong maximum principle, called \textit{a positive solution}. 
In fact, Hopf's boundary point lemma is difficult to apply to the positivity on $\partial\Omega$ for a positive solution of \eqref{p}, because it is unclear from $0<q<1$ if it is $C^1(\overline{\Omega})$. 
If a positive solution $u$ of \eqref{p} is positive on the entirety of $\partial \Omega$, then $u \in C^{2+\theta}(\overline{\Omega})$ by a bootstrap argument based on elliptic regularity, and thus, $u$ satisfies \eqref{p} {\it pointwisely} in $\overline{\Omega}$ (actually, it is verified by Proposition \ref{prop:uniq-spositive}(ii) below that a positive solution of \eqref{p} is indeed positive in $\overline{\Omega}$, and so, it is $C^{2+\theta}(\overline{\Omega})$). 

Problem \eqref{p} with $\mu=0$, i.e., the Neumann problem has the unique positive solution $u=1$. We know that, for $\mu > 0$, problem \eqref{p} has {\it at most} one positive solution that is positive in $\overline{\Omega}$ (\cite[Theorem 6.3, Chapter 4]{Pa92}). The objective of this paper is to evaluate the unique existence of a positive solution for \eqref{p} and its asymptotic profile as  $\mu \to \infty$. Our main result below presents that the unique positive solution converges to a boundary layer solution of the associated boundary blow--up problem. 
%
%
\begin{theorem} \label{mr} 
Problem \eqref{p} has a {\rm unique} positive solution $u_{\mu}$ for each $\mu > 0$, satisfying $u_{0}=1$, and $\mu \mapsto u_{\mu}$ is an increasing $C^{\infty}$--map from $[0, \infty)$ into $C^{2+\theta}(\overline{\Omega})$, $0<\theta<1$. Moreover, the following {\rm three} assertions hold:
\begin{enumerate}
\item There exists $u_{\infty}\in C^{2+\theta}(\Omega)$ such that, for $S\Subset \Omega$, $u_{\mu} \rightarrow u_\infty$ in $C^{2+\theta}(\overline{S})$ as $\mu\to \infty$. Moreover, $u_{\infty}$ satisfies 
\begin{align} \label{limitprob}
\begin{cases}
-\Delta u = u-u^{p} & \mbox{ in } \Omega, \\
u>0 & \mbox{ in } \Omega, \\ 
u = \infty & \mbox{ on } \partial\Omega.  
\end{cases} 
\end{align}
Here, the statement $u=\infty$ on $\partial\Omega$ means that $u(x) \rightarrow \infty$ as $d(x, \partial\Omega) \to 0$. More precisely, there exist $\overline{A}\geq \underline{A}>0$ such that, if $\mu > 0$ is sufficiently large, then we obtain  
\begin{align*}
\underline{A}\, \mu^{\frac{2}{p-2q+1}} \leq u_{\mu} \leq \overline{A}\, \mu^{\frac{2}{p-2q+1}} \quad\mbox{ on } \partial\Omega. 
\end{align*}

\item For $r\in \mathbb{R}$, $\int_{\Omega} u_{\mu}^{r} dx \rightarrow \int_{\Omega} u_{\infty}^{r} dx \in (0, \infty]$ as $\mu \to \infty$. 
If $r<\frac{p-1}{2}$, then $\int_{\Omega} u_{\infty}^{r} dx \in (0, \infty)$, and further, $\int_{\Omega} (u_{\infty}-u_{\mu})^{r}dx \rightarrow 0$ as $\mu \to \infty$ additionally if $r>0$. Meanwhile, if $r\geq \frac{p-1}{2}$, then $\int_{\Omega} u_{\infty}^{r}dx = \infty$. 

\item Let $r\geq1$. Assume additionally that $\Omega$ is convex if $r\geq \frac{p-1}{2}$. Then, $\| \nabla u_{\mu} \|_{L^{r}(\Omega)}\rightarrow \infty$ as $\mu \to \infty$. 

\end{enumerate} 
\end{theorem}

\begin{rem} 
\begin{enumerate}
    \item We know that problem \eqref{limitprob} has at most one positive solution (see for instance \cite[Theorem 2.2]{MV97}).  
    \item When $q=1$, Garc\'ia-Meli\'an, Sabina de Lis, and Rossi \cite[Theorem 4]{GM07} have obtained a similar result to assertion (i). 
\end{enumerate}
\end{rem}

Few studies have explored the existence of a boundary layer profile for the case of nonlinear boundary conditions with parameters. This study is inspired by Garc\'ia-Meli\'an, Rossi, and Sabina de Lis \cite{GM10}, who studied layer profiles as $\lambda \to \infty$ of the unique positive solution for the equation $\Delta u = u^{p}$ in $\Omega$, $p>1$, equipped with the nonlinear boundary condition $\frac{\partial u}{\partial \nu}=\lambda u - u^q$ on $\partial \Omega$, $q>1$ (see \cite[Theorem 1.4]{GM10}). 
For the existence, uniqueness, and multiplicity of positive solutions for logistic elliptic equations with the nonlinear boundary condition $\frac{\partial u}{\partial \nu} = u^q$ on $\partial\Omega$, $q>0$, we refer to \cite{MR05, GM08}.

It should be emphasized that the sublinear boundary condition appearing in \eqref{p} models coastal fishery harvesting when $\mu=-\lambda<0$ (see \cite[Subsection 2.1]{GUU19}). In fact, when $p=2$, the unknown function $u>0$ in $\Omega$ for \eqref{p} ecologically represents the biomass of fish inhabiting, for instance, a \textit{lake} $\Omega$, which obeys the logistic law. Then, the sublinear boundary condition represents fishery harvesting with the harvesting effort $\lambda$ on the \textit{lake coast} $\partial\Omega$. 
In \cite{Um2023, Um2024, Um2024b}, the author studied the existence, uniqueness, and multiplicity of positive solutions for \eqref{p} with $\mu<0$ under the additional assumption $p<\frac{N+2}{N-2}$ if $N>2$ and also studied their asymptotic profiles as $\mu \to 0^{-}$ and $\mu \to -\infty$. Indeed, Theorem \ref{mr} compliments previous results obtained for $\mu<0$ (see Figures \ref{fig_betaless1}--\ref{fig_betagtr1}), where $\beta_{\Omega}$ is the positive value defined as in \eqref{Deprob}, and $u_{\mathcal{D}}\in C^{2+\theta}(\overline{\Omega})$ is the unique positive solution of the Dirichlet problem for the case of $\beta_{\Omega}<1$ (\cite{BO86}) 
\begin{align*}
    \begin{cases}
        -\Delta u = u-u^{p} & \mbox{ in } \Omega, \\
        u = 0 & \mbox{ on } \partial\Omega.  
    \end{cases}
\end{align*}
It is understood that $u_{\mathcal{D}}=0$ when $\beta_{\Omega}=1$. 
%
    \begin{figure}[!htb]
    \centering 
    \includegraphics[scale=0.18]{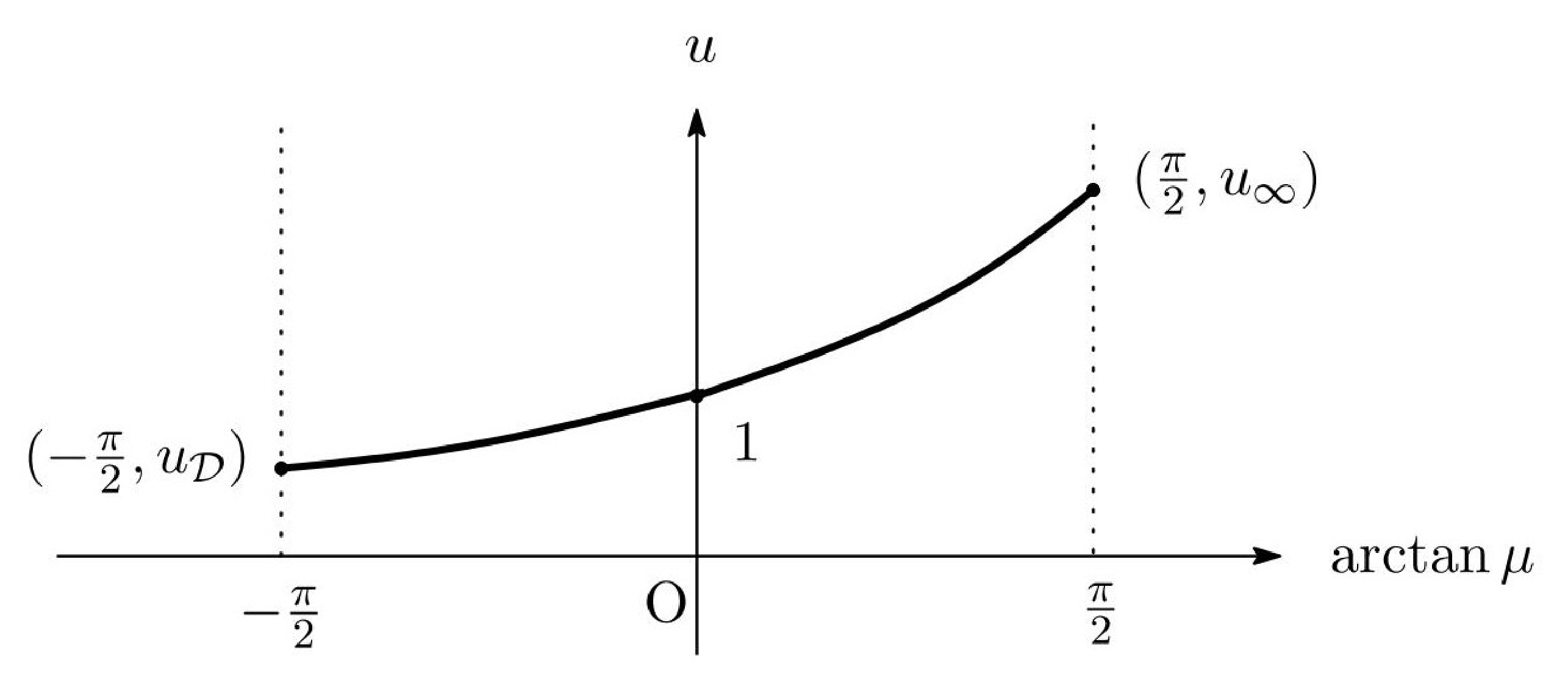} 
		  \caption{Possible positive solution set in the case when $\beta_{\Omega}<1$, or $\beta_{\Omega}=1$ and $pq>1$: no bifurcation points of positive solutions for \eqref{p} on $\{(\mu, 0) : \mu \in \mathbb{R} \}$.}
		\label{fig_betaless1} 
    \end{figure} 
%
    \begin{figure}[!htb]
    \centering 
    \includegraphics[scale=0.18]{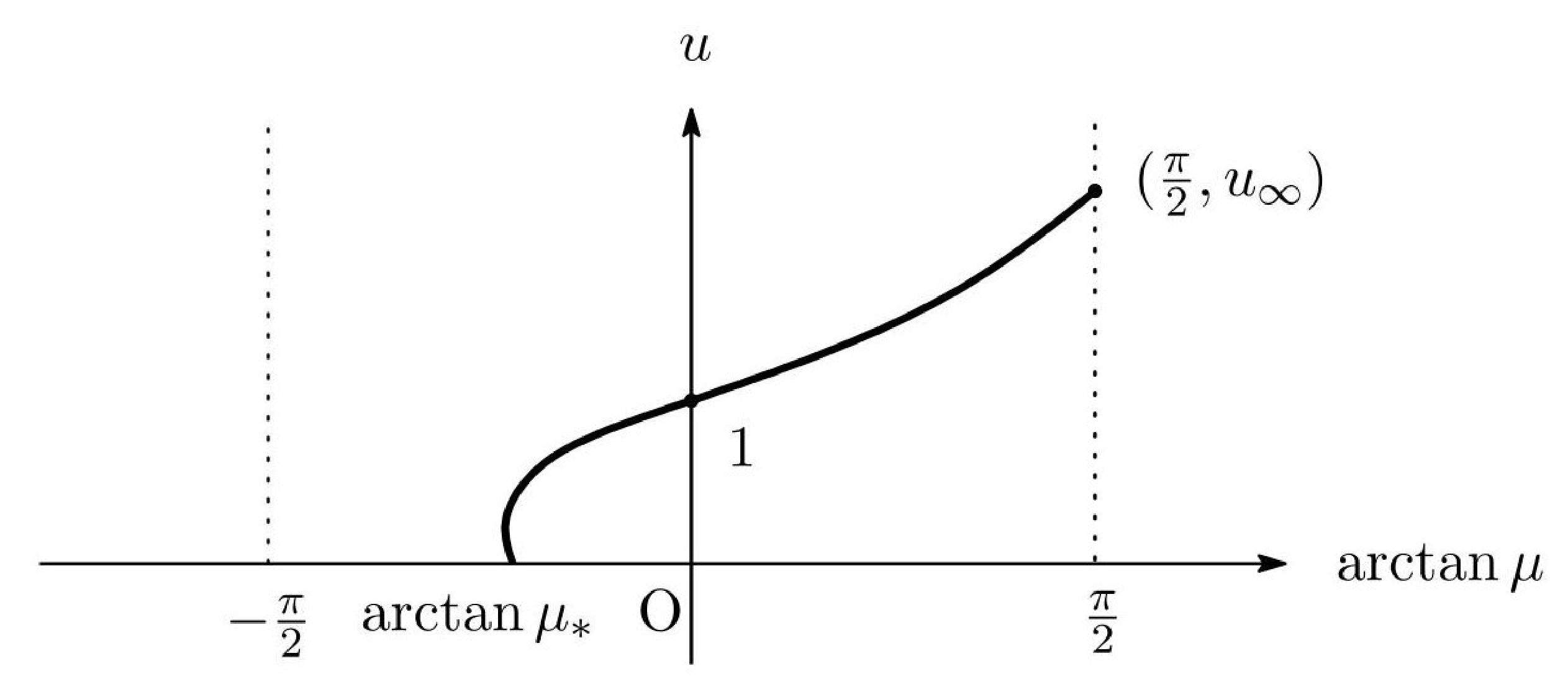} 
		  \caption{Possible positive solution set in the case when $\beta_{\Omega}=1$ and $pq=1$: a bifurcation point of positive solutions for \eqref{p} on and only on $\{ (\mu, 0) : \mu < 0 \}$.}  
		\label{fig_pq1} 
    \end{figure} 
%
    \begin{figure}[!htb]
    \centering 
    \includegraphics[scale=0.18]{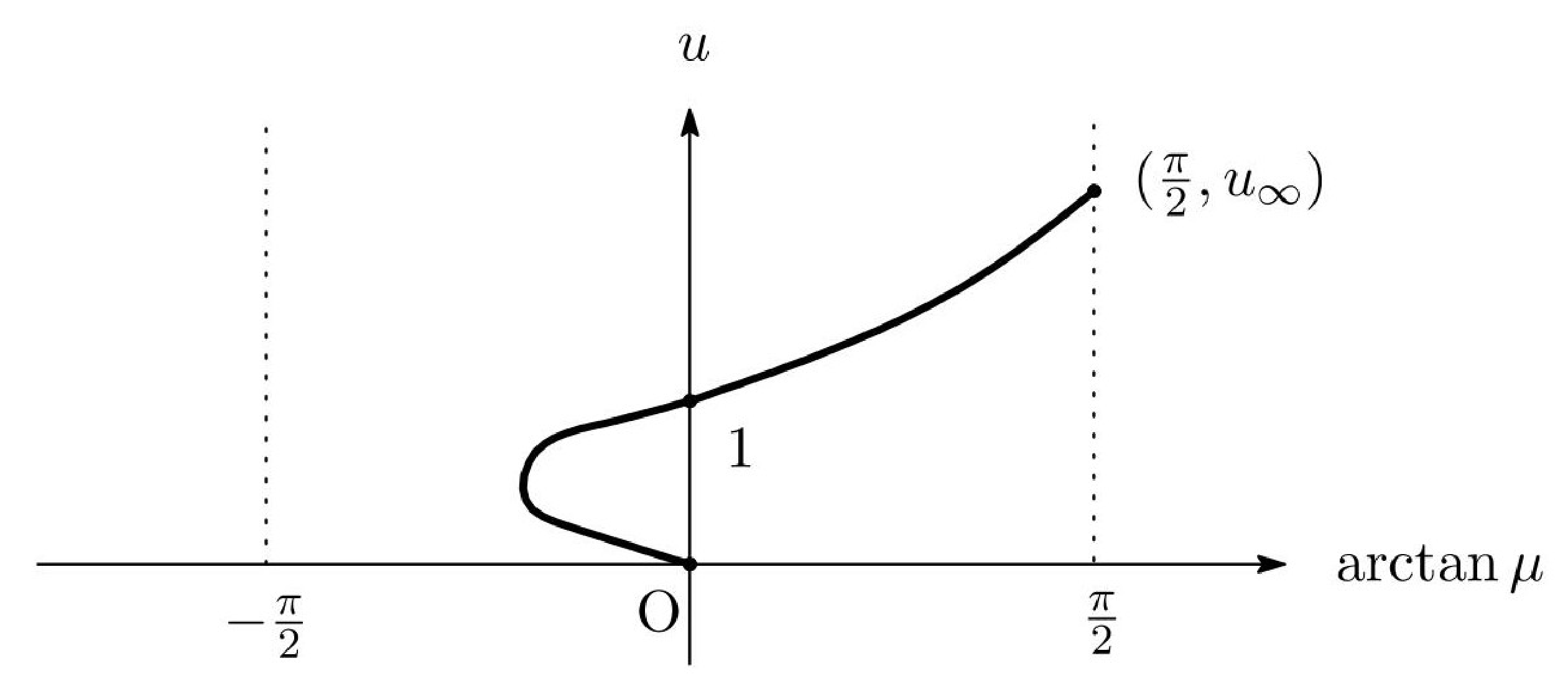} 
		  \caption{Possible positive solution set in the case when $\beta_{\Omega}=1$ and $pq<1$, or $\beta_{\Omega}>1$: a bifurcation point of positive solutions for \eqref{p} at and only at $(\mu, u)=(0,0)$.}
		\label{fig_betagtr1} 
    \end{figure}

The remainder of this paper is organized as follows. Section \ref{sec:bounds} is devoted to the construction of upper and lower bounds of a positive solution for \eqref{p}, which is a crucial step for our aim. Section \ref{sec:curve} proves the existence of a positive solution curve for \eqref{p} emanating from $(\mu,u) = (0,1)$. In Section \ref{sec:profile}, we prove Theorem \ref{mr}, where the asymptotic behavior of the positive solution curve is examined as $\mu \to \infty$, and then, a boundary layer solution of \eqref{limitprob} is established for the limiting profile.

\section{Upper and lower bounds of positive solutions} 
\label{sec:bounds}

First, we discuss the uniqueness of a positive solution for \eqref{p} and its positivity on $\partial\Omega$. To this end, we use a special case of the comparison principle \cite[Proposition A.1]{RQU2017b}, presented in Lemma \ref{prop:cp} below, which is an extension of \cite[Lemma 3.3]{ABC94} to the case of concave nonlinear boundary conditions. 
\begin{lem} \label{prop:cp}
Let $f, g: [0,\infty) \rightarrow \mathbb{R}$ be continuous. Assume that $t\mapsto \frac{f(t)}{t}$ is decreasing in $(0, \infty)$ and 
$t\mapsto \frac{g(t)}{t}$ is nonincreasing in $(0,\infty)$. 
Let $u, v \in H^{1}(\Omega) \cap C(\overline{\Omega})$ be nonnegative functions satisfying 
\begin{align}
&\int_{\Omega} \nabla u \nabla \varphi \, dx - \int_{\Omega} f(u) \varphi \, dx - \int_{\partial\Omega} g(u)\varphi \, d\sigma \leq 0, \quad \forall \varphi \in H^{1}(\Omega) \ \mbox{ and } \ \varphi \geq 0 \ \mbox{ a.e. in } \ \Omega, \label{w:sub}\\
& \int_{\Omega} \nabla v \nabla \varphi \, dx - \int_{\Omega} f(v) \varphi \, dx - \int_{\partial\Omega} g(v)\varphi \, d\sigma \geq 0, \quad \forall \varphi \in H^{1}(\Omega) \ \mbox{ and } \ \varphi \geq 0 \ \mbox{ a.e. in } \ \Omega. \label{w:super}
\end{align}
If $v>0$ in $\Omega$, then $u\leq v$ in $\overline{\Omega}$. 
\end{lem}

The next proposition is a direct consequence of Lemma \ref{prop:cp}. 
\begin{prop} \label{prop:uniq-spositive} 
The following two assertions are valid: 
\begin{enumerate}
    \item A positive solution of \eqref{p} is unique. 
    \item A positive solution $u$ of \eqref{p} is such that $u\geq 1$ in $\overline{\Omega}$. 
\end{enumerate}
\end{prop}

\begin{proof}
(i) Let $u, v$ be positive solutions of \eqref{p}. Note that $u, v > 0$ in $\Omega$. Then, because $u, v$ both satisfy \eqref{sol:def}, applying Lemma \ref{prop:cp} with $f(u)=u-u^{p}$ and $g(u)=\mu u^{q}$ provides $u\leq v$ and $v\leq u$, as desired. 

(ii) Observe that $\underline{u}=1$ verifies  
\begin{align*} 
\int_{\Omega} \biggl( \nabla \underline{u} \nabla \varphi - \underline{u}\varphi + \underline{u}^{p} \varphi \biggr)dx - \mu \int_{\partial\Omega} \underline{u}^{q} \varphi \,  d\sigma \leq 0, \quad \forall \varphi \in H^1(\Omega) \ \mbox{ and } \ \varphi\geq 0 \ \mbox{ a.e. in } \ \Omega. 
\end{align*}
Similarly, we deduce that $u\geq 1$ in $\overline{\Omega}$ for a positive solution $u$ of \eqref{p}. 
\end{proof}

Second, in terms of Lemma \ref{prop:cp}, we construct sub-- and supersolutions of \eqref{p} to establish {\it a priori} upper and lower bounds of a positive solution for \eqref{p}. A function $\underline{u}\in C^{2+\theta}(\overline{\Omega})$ with the condition that $\underline{u}>0$ in $\overline{\Omega}$ is said to be a {\it subsolution} of \eqref{p} if $\underline{u}$ satisfies   
\begin{align*}
\begin{cases}
-\Delta u \leq u-u^{p} & \mbox{ in } \Omega,  \\
\frac{\partial u}{\partial \nu}  \leq \mu u^q & \mbox{ on } \partial\Omega. 
\end{cases} 
\end{align*}
A supersolution of \eqref{p} is defined by reversing the inequalities (\cite[Theorem 2.1]{Am76N}). 
The next result is clear from Lemma \ref{prop:cp}.  
\begin{lem} \label{lem:compa2}
If $\underline{u}$ and $\overline{u}$ are sub- and supersolutions of \eqref{p}, respectively, then $\underline{u}\leq \overline{u}$ in $\overline{\Omega}$. 
\end{lem}

We denote by $\phi_{\Omega} \in C^{2+\theta}(\overline{\Omega})$ the positive eigenfunction, satisfying $\| \phi_{\Omega}\|_{C(\overline{\Omega})}=1$, associated with the smallest eigenvalue $\beta_{\Omega}>0$ of the Dirichlet eigenvalue problem 
\begin{align} \label{Deprob}
    \begin{cases}
        -\Delta \phi = \beta \phi & \mbox{ in } \Omega, \\
        \phi = 0 & \mbox{ on }  \partial \Omega.  
    \end{cases}
\end{align}
Choose $c_{1}, c_{2}, c_{3}>0$ such that 
$c_{1} \leq -\frac{\partial \phi_{\Omega}}{\partial \nu}\leq c_{2}$ on $\partial\Omega$ and $|\nabla \phi_{\Omega} |^2 \leq c_{3}$ in $\Omega$. Let 
\begin{align} \label{Sep}
S_{\varepsilon}=\{ x\in \Omega : d(x,\partial\Omega)>\varepsilon \}, \quad \varepsilon>0. 
\end{align}
In terms of the presence of $c_{1}$, we can choose $b_{1}>0$ and a sufficiently small $\varepsilon_{0}>0$ such that $|\nabla \phi_{\Omega}|^2 \geq b_{1}$ in $\Omega \setminus \overline{S_{\varepsilon_{0}}}$, and then, we say $b_{2} = \inf_{S_{\varepsilon_{0}}} \phi_{\Omega}>0$. We now define 
\begin{align} \label{def:psiA}
\psi = \psi_{A}(x) = A(\phi_{\Omega}(x) + \mu^{-\tau})^{-\rho}, \quad x 
\in \overline{\Omega}, 
\end{align}
where $\tau = \frac{p-1}{p-2q+1}$, $\rho=\frac{2}{p-1}$, and $A>0$ is a constant that will be determined later. Note that $0<q<1<p$ implies $p-2q+1>0$; thus, $\tau, \rho > 0$. For $\underline{\mu}>0$, take $c_{\underline{\mu}}>0$ such that 
$\phi_{\Omega}+\mu^{-\tau}\leq c_{\underline{\mu}}$ in $\Omega$ for $\mu\geq \underline{\mu}$. 

The next lemma is crucial for obtaining upper and lower bounds of a positive solution for \eqref{p}. 
\begin{lem} \label{lem:subsuper}
Let $\underline{\mu}>0$. Then, there exist $\underline{A}, \overline{A}>0$ such that if $\mu\geq \underline{\mu}$, then $\psi_{A}$ is a subsolution of \eqref{p} for $A\leq \underline{A}$, while it is a supersolution of \eqref{p} for $A\geq \overline{A}$. Here, $\underline{A}$ and $\overline{A}$ do not depend on $\mu \in [\underline{\mu}, \infty)$ varying (consequently, $\underline{A}\leq \overline{A}$ from 
Lemma \ref{lem:compa2}). 
\end{lem}

\begin{proof}
We construct a supersolution of \eqref{p} in terms of \eqref{def:psiA}. By simple calculations, 
$\nabla \psi = A(-\rho)(\phi_{\Omega} + \mu^{-\tau})^{-\rho-1}\nabla \phi_{\Omega}$, and thus, 
\begin{align*}
\Delta \psi =A(-\rho)(-\rho-1)(\phi_{\Omega} + \mu^{-\tau})^{-\rho-2}|\nabla \phi_{\Omega}|^2 
-A(-\rho)(\phi_{\Omega}+\mu^{-\tau})^{-\rho -1} \beta_{\Omega} \phi_{\Omega}.   
\end{align*}
Since $(1-p)\rho +2=0$, it follows that 
\begin{align} 
\Delta \psi + \psi - \psi^{p} 
&=A(\phi_{\Omega} + \mu^{-\tau})^{-\rho-2}\biggl\{ \rho(\rho +1) |\nabla \phi_{\Omega}|^2 
+ \rho (\phi_{\Omega} + \mu^{-\tau})\beta_{\Omega} \phi_{\Omega} 
+ (\phi_{\Omega} + \mu^{-\tau})^{2} - A^{p-1} \biggr\} \nonumber \\ 
&=: A(\phi_{\Omega} + \mu^{-\tau})^{-\rho-2} I_{A}.  \label{psiformula}
\end{align}
Then, we deduce that, for $\mu\geq \underline{\mu}$,  
\begin{align*}
I_{A} \leq \rho(\rho +1)\, c_{3}+ 
\rho c_{\underline{\mu}} \beta_{\Omega} + c_{\underline{\mu}}^{\, 2} - A^{p-1} \leq 0 \quad\mbox{ in } \Omega, 
\end{align*}
whenever $A\geq \overline{A}_{1}:= \left\{  
\rho(\rho +1)\, c_{3} + \rho c_{\underline{\mu}}\beta_{\Omega} + c_{\underline{\mu}}^{\, 2}
\right\}^{\frac{1}{p-1}}$. 
Furthermore, considering $\tau (\rho +1)=1+ \tau \rho q = \frac{p+1}{p-2q+1}$, 
we infer that 
\begin{align}
\frac{\partial \psi}{\partial \nu} - \mu \psi^q 
&= A^{q} \mu^{\tau (\rho +1)}\left\{ A^{1-q}\rho \left( -\frac{\partial \phi_{\Omega}}{\partial \nu} \right) -1 \right\} \label{psiformula2}  \\  
&\geq A^q\mu^{\tau (\rho +1)}(A^{1-q}\rho \, c_{1}-1) \geq 0 \quad\mbox{ on } \partial\Omega,\nonumber 
\end{align}
whenever $A \geq \overline{A}_{2}:= \left( \rho \, c_{1} \right)^{\frac{1}{q-1}}$. 
Thus, $\overline{A}:=\max(\overline{A}_{1}, \overline{A}_{2})$ is as desired. 

Then, we construct a subsolution of \eqref{p} in terms of \eqref{def:psiA}. 
From \eqref{psiformula}, we deduce  
\begin{align*}
I_{A} \geq \rho (\rho +1)|\nabla \phi_{\Omega}|^2 - A^{p-1}\geq \rho (\rho +1)\, b_{1} - A^{p-1}\geq 0 \quad\mbox{ in } \ 
\Omega\setminus \overline{S_{\varepsilon_{0}}}, 
\end{align*}
whenever $A\leq \underline{A}_{1}:=\{ \rho (\rho +1)\, b_{1} \}^{\frac{1}{p-1}}$. 
Similarly, we deduce 
\begin{align*}
I_{A} \geq (\phi_{\Omega} + \mu^{-\tau})^{2} - A^{p-1} \geq b_{2}^{\, 2} - A^{p-1}\geq 0 \quad\mbox{ in } \ 
S_{\varepsilon_{0}}, 
\end{align*}
whenever $A\leq \underline{A}_{2}:=b_{2}^{\, \frac{2}{p-1}}$. Furthermore, from \eqref{psiformula2}, we infer 
\begin{align*}
\frac{\partial \psi}{\partial \nu} - \mu \psi^q 
\leq A^{q}\mu^{\tau (\rho +1)}(A^{1-q}\rho \, c_{2} - 1)\leq 0 \quad \mbox{ on } \partial \Omega, 
\end{align*}
whenever $A\leq \underline{A}_{3}:=\left( \rho \, c_{2} \right)^{\frac{1}{q-1}}$. 
Thus, $\underline{A}:=\min(\underline{A}_{1}, \underline{A}_{2}, \underline{A}_{3})$ is as desired. 
\end{proof}

Using Lemma \ref{lem:compa2}, we deduce from Lemma \ref{lem:subsuper} the following upper and lower bounds of a positive solution for \eqref{p}. 
\begin{prop} \label{prop:bounds}
Let $\underline{\mu}>0$, and let $\underline{A}, \overline{A}$ be as determined in Lemma \ref{lem:subsuper}. Then, we obtain 
\begin{align} \label{Aline} 
\underline{A}\, (\phi_{\Omega} + \mu^{-\frac{p-1}{p-2q+1}} )^{-\frac{2}{p-1}}\leq u \leq 
\overline{A}\, (\phi_{\Omega} + \mu^{-\frac{p-1}{p-2q+1}} )^{-\frac{2}{p-1}} \quad 
\mbox{ in } \overline{\Omega}
\end{align}
for a positive solution $u$ of \eqref{p} with $\mu\geq \underline{\mu}$. 
\end{prop}

The next corollary follows directly from Proposition \ref{prop:bounds}.  
\begin{cor}
Let $\underline{\mu}>0$. Then, for $S_{\varepsilon}$ defined by \eqref{Sep},  
\begin{align} \label{boundonS} 
u\leq \overline{A} \, \biggl( \inf_{S_{\varepsilon}}\phi_{\Omega} \biggr)^{-\frac{2}{p-1}} \quad\mbox{ in } \ S_{\varepsilon}, 
\end{align}
and moreover, 
\begin{align} \label{boundonbdry}
\underline{A}\, \mu^{\frac{2}{p-2q+1}} \leq u \leq \overline{A}\, \mu^{\frac{2}{p-2q+1}} 
\quad\mbox{ on } \partial\Omega 
\end{align}
for a positive solution $u$ of \eqref{p} with $\mu\geq \underline{\mu}$. 
\end{cor}

\section{Existence of a positive solution curve}  
\label{sec:curve} 

In this section, we consider \eqref{p} in the framework of H\"older spaces. 
Let $(\lambda, v)$, $\lambda\geq0$, be a positive solution of \eqref{p}, let 
$\delta_{1}, \delta_{2} > 0$, and let $U = \{ u \in C^{2+\theta}(\overline{\Omega}): \| u -v\|_{C^{2+\theta}(\overline{\Omega})} < \delta_2 \}$. 
In terms of Proposition \ref{prop:uniq-spositive}(ii), we see that $u>0$ in $\overline{\Omega}$ for $u \in U$ if $\delta_{2}>0$ is small. 
Thus, it is possible to define the map $F: (\lambda-\delta_{1}, \lambda+\delta_{1})\times U \rightarrow C^{\theta}(\overline{\Omega})\times C^{1+\theta}(\partial\Omega)$ as $F(\mu,u)=(-\Delta u - u + u^p, \frac{\partial u}{\partial \nu}-\mu u^q)$. To solve \eqref{p}, we consider the equation $F(\mu, u)=0$. 

The Fr\'echet derivative $F_u$ of $F$ with respect to $u$ at $(\lambda, v)$ is given as $F_{u}(\lambda, v)\varphi = ( -\Delta \varphi - \varphi + pv^{p-1}\varphi, \frac{\partial \varphi}{\partial \nu} - q \lambda v^{q-1} \varphi )$. For the associated eigenvalue problem
\begin{align*} 
\begin{cases}
(-\Delta - 1 + pv^{p-1})\varphi = \gamma \varphi & \mbox{ in } \Omega, \\
(\frac{\partial}{\partial \nu}  - q \lambda v^{q-1}) \varphi = \gamma \varphi & \mbox{ on } \partial\Omega, 
\end{cases} 
\end{align*}
where $\gamma\in \mathbb{R}$ is an eigenvalue parameter, we deduce that the smallest eigenvalue $\gamma_{1}=\gamma_{1}(\lambda, v)$ is positive.  Indeed, a positive eigenfunction $\varphi_1$ corresponding to $\gamma_1$ provides 
\begin{align} \label{gamm1posi}
\gamma_{1} = \frac{(p-1)\int_{\Omega} v^q \varphi_1 dx + \lambda (1-q) \int_{\partial\Omega}v^q \varphi_1 d\sigma}{\int_{\Omega} v\varphi_1 dx + \int_{\partial\Omega} v\varphi_1 d\sigma}>0.     
\end{align}
We then apply the implicit function theorem to deduce that, in a neighborhood of $(\lambda, v)$, $F(\mu, u)=0$ consists of a $C^{\infty}$--positive solution curve $\{ (\mu, u_{\mu}) : \mu \in (\lambda-\delta, \lambda+\delta)\}$ for some $\delta > 0$ with the condition that $u_{\lambda}=v$, $u_{\mu}>0$ in $\overline{\Omega}$, and the map $(\lambda -\delta, \lambda +\delta) \rightarrow U$; $\mu\mapsto u_{\mu}$ is $C^{\infty}$. Then, differentiating \eqref{p} with $u=u_{\mu}$ by $\mu$ yields 
\begin{align} \label{eq:umuprime} 
\begin{cases}
\left( -\Delta -1 + pu_{\mu}^{p-1} \right)\frac{d u_{\mu}}{d \mu} = 0 & \mbox{ in } \Omega,  \\
\left( \frac{\partial}{\partial \nu} -q \mu u_{\mu}^{q-1} \right) 
\frac{d u_{\mu}}{d \mu} = u_{\mu}^{\, q} > 0 & \mbox{ on } \partial\Omega. 
\end{cases} 
\end{align}
In terms of \cite[Theorem 7.10]{LGbook2}, assertions \eqref{gamm1posi} and \eqref{eq:umuprime} provide 
\begin{align} \label{umupriposi}
\frac{d u_{\mu}}{d \mu} > 0 \quad \mbox{ in } \ \overline{\Omega}. 
\end{align}

Then, using \eqref{gamm1posi} and \eqref{umupriposi}, we prove the following existence result. 
\begin{prop} \label{prop:curve}
There exists a $C^{\infty}$--positive solution curve $\{ (\mu, u_{\mu})\in \mathbb{R}\times C^{2+\theta}(\overline{\Omega}): \mu \geq 0 \}$ of \eqref{p} such that $u_{0}=1$, and $\mu \longmapsto u_{\mu}$ is $C^\infty$ and increasing, i.e., $u_{\mu}<u_{\tilde{\mu}}$ in $\overline{\Omega}$ for $\mu < \tilde{\mu}$. Furthermore, the positive solution set of \eqref{p} is represented by the positive solution curve in a neighborhood of $(\mu, u_{\mu})$. 
\end{prop}

\begin{proof}
Observe $F(0,1)=0$. On the basis of \eqref{gamm1posi} and \eqref{umupriposi}, we then use the implicit function theorem step--by--step to obtain a positive solution curve $\{ (\mu, u_{\mu})\in \mathbb{R}\times C^{2+\theta}(\overline{\Omega}): 0\leq \mu < \overline{\mu} \}$ of \eqref{p} for some $\overline{\mu}\in (0, \infty]$ such that $u_{0}=1$, $\mu \mapsto u_{\mu}$ is $C^{\infty}$ and increasing, and $\| u_{\mu} \|_{C^{2+\theta}(\overline{\Omega})} \rightarrow \infty$ as $\mu \to \overline{\mu}^{-}$. In addition, the positive solution set of \eqref{p} can be represented by the positive solution curve in a neighborhood of $(\mu, u_{\mu})$. 
Then, we verify $\overline{\mu}=\infty$.  Assume by contradiction that $\overline{\mu} <\infty$. Choosing $\underline{\mu}\in (0, \overline{\mu})$, we deduce from \eqref{Aline} that $u_{\mu}\leq \overline{A} \, \overline{\mu}^{\frac{2}{p-2q+1}}$ in $\overline{\Omega}$ for $\mu\in [\underline{\mu}, \overline{\mu})$. Hence, for $1<r<\infty$, we apply the $W^{1,r}$estimate \cite[Proposition 3.3]{Am76N} to $(\mu, u_{\mu})$ to obtain that $\| u_{\mu}\|_{W^{1,r}(\Omega)}\leq C$ for $\mu \in [0, \overline{\mu})$. Here, and in what follows, $C$ is a generic positive constant independent of $\mu$, which may change from line to line. For such $\mu$, 
$\| u_{\mu}\|_{W^{2,r}(\Omega)}\leq C$ by elliptic regularity, which implies that $\| u_{\mu}\|_{C^{1+\theta}(\overline{\Omega})}\leq C$ by Sobolev's embedding theorem. Therefore, Schauder's estimate yields that $\| u_{\mu}\|_{C^{2+\theta}(\overline{\Omega})}\leq C$ for $\mu \in [0, \overline{\mu})$, which is a contradiction. 
\end{proof}

\section{Asymptotic profiles and existence of a boundary layer solution} 
\label{sec:profile} 

In this section, we give a proof for Theorem \ref{mr}. 
Let $(\mu, u_{\mu})$, $\mu \geq 0$, be the unique positive solution of \eqref{p} ensured by Propositions \ref{prop:uniq-spositive}(i) and \ref{prop:curve}. Since $u_{\mu}$ is increasing for $\mu \geq 0$, assertion \eqref{boundonS} deduces that $u_{\mu}(x)$ converges pointwisely to some $u_{\infty}(x)$ for $x \in \Omega$ as $\mu \to \infty$. 

The next proposition presents that $u_{\infty}$ provides the asymptotic boundary layer profile for $u_{\mu}$ as $\mu \to \infty$. 
\begin{prop} \label{prop:layer}
Let $(\mu, u_{\mu})$, $\mu \geq 0$, and $u_{\infty}$ be as above. 
Then, for $\varepsilon > 0$, 
\begin{align} \label{umutouinfty}
u_{\mu} \longrightarrow u_{\infty} \ \mbox{ in } C^{2+\theta}(\overline{S_{\varepsilon}}) \ \mbox{ as } \mu \to \infty.
\end{align}
Further, $u_{\infty}\in C^{2+\theta}(\Omega)$ satisfies \eqref{limitprob}. 
\end{prop}

\begin{proof}
Let $1<r<\infty$, let $\mu\geq \underline{\mu}$, and let $0<\varepsilon^{\prime\prime}<\varepsilon^{\prime}<\varepsilon$. Note that $S_{\varepsilon}\Subset S_{\varepsilon^{\prime}}\Subset S_{\varepsilon^{\prime\prime}}\Subset \Omega$. 
In terms of \cite[Theorem 9.11]{GT83}, we obtain  
\begin{align*}
\| u_{\mu}\|_{W^{2,r}(S_{\varepsilon^{\prime}})} \leq C\left( \| u_{\mu}\|_{L^{r}(S_{\varepsilon^{\prime\prime}})} 
+ \| u_{\mu} - u_{\mu}^p \|_{L^{r}(S_{\varepsilon^{\prime\prime}})} \right).  
\end{align*}
Since $\| u_{\mu}\|_{L^{\infty}(S_{\varepsilon^{\prime\prime}})}\leq C$ from \eqref{boundonS}, we deduce that $\| u_{\mu} \|_{W^{2,r}
(S_{\varepsilon^{\prime}})}\leq C$, implying $\| u_{\mu} \|_{C^{1+\theta}(\overline{S_{\varepsilon^{\prime}}})}\leq C$ by Sobolev's embedding theorem. Furthermore, in terms of \cite[Theorem 6.2]{GT83}, we obtain 
\[
\| u_{\mu} \|_{C^{2+\theta}(\overline{S_{\varepsilon}})} \leq C \left( \| u_{\mu} \|_{C(\overline{S_{\varepsilon^{\prime}}})} 
+ \| u_{\mu}-u_{\mu}^p \|_{C^{\theta}(\overline{S_{\varepsilon^{\prime}}})} \right), 
\]
and thus, it follows that $\| u_{\mu}\|_{C^{2+\theta}(\overline{S_{\varepsilon}})}\leq C$. By a compactness argument, we deduce that, up to a subsequence, $u_{\mu} \rightarrow u_{\infty}$ in $C^{2+\alpha}(\overline{S_{\varepsilon}})$ as $\mu \to \infty$ for $\alpha < \theta$. 
Assertion \eqref{umutouinfty} thus follows, because $u_{\infty}$ is unique, and $\theta$ and $\alpha$ are arbitrary. 
Finally, assertions \eqref{boundonbdry} and \eqref{umutouinfty} ensure that $u_{\infty} \in C^{2+\theta}(\Omega)$, and it admits \eqref{limitprob}. 
\end{proof}

\begin{proof}[Proof of Theorem \ref{mr}]
The existence and uniqueness assertions are from Propositions \ref{prop:curve} and \ref{prop:uniq-spositive}(i), respectively. 
Assertion (i) follows from Proposition \ref{prop:layer} and \eqref{boundonbdry}. 

Now, we verify assertion (ii). Since $u_{\mu}(x) \rightarrow u_{\infty}(x)$ for $x \in \Omega$ and $u_{\mu}$ is increasing for $\mu \geq 0$, the monotone convergence theorem provides that, for $r>0$, 
\begin{align*}
\int_{\Omega} u_{\mu}^{r}dx \longrightarrow \int_{\Omega} u_{\infty}^{r} dx \in (0, \infty]. 
\end{align*}
First, we consider the case $0< r < \frac{p-1}{2}$.
Choose $d_{1}>0$ such that $d_1 \leq \frac{\phi_{\Omega}(x)}{d(x, \partial\Omega)}$ for $x \in \Omega$. From \eqref{Aline}, we deduce that 
\[
\int_{\Omega} u_{\mu}^{r} dx 
\leq \overline{A}^{\, r} \int_{\Omega} \phi_{\Omega}^{-\frac{2r}{p-1}} dx 
\leq \overline{A}^{\, r} d_{1}^{-\frac{2r}{p-1} }\int_{\Omega} d(x,\partial\Omega)^{-\frac{2r}{p-1}} dx. 
\]
Because $\int_{\Omega} d(x,\partial\Omega)^{-\frac{2r}{p-1}} dx < \infty$ from $-\frac{2r}{p-1}>-1$, we deduce that $\int_{\Omega} u_{\infty}^{r} dx \in (0, \infty)$. In addition, because $0<(u_{\infty}-u_{\mu})^{r}\leq u_{\infty}^{r}$ in $\Omega$, the Lebesgue dominated convergence theorem provides that $\int_{\Omega} (u_{\infty}-u_{\mu})^{r} dx \rightarrow 0$ as $\mu \to \infty$. 
Second, we consider the case $r\geq \frac{p-1}{2}$. Choose $d_2 > 0$ such that $\frac{\phi_{\Omega}(x)}{d(x,\partial\Omega)}\leq d_{2}$ for $x\in \Omega$, and 
we set $C_{r} = \underline{A}^{r}2^{-\frac{2r}{p-1}}d_{2}^{-\frac{2r}{p-1}}>0$. 
Let $K>0$. By the condition $-\frac{2r}{p-1} \leq -1$, there exists a small $\varepsilon > 0$ such that $\int_{S_{\varepsilon}} d(x,\partial\Omega)^{-\frac{2r}{p-1}} dx > \frac{K}{C_{r}}$, and then, we can choose a large $\mu_{\varepsilon}>0$ such that if $\mu\geq \mu_{\varepsilon}$, then $\mu^{-\frac{p-1}{p-2q+1}}\leq \phi_{\Omega}(x)$ for $x \in S_{\varepsilon}$. Thus, we deduce from \eqref{Aline} that if $\mu\geq \mu_{\varepsilon}$, then 
\begin{align*}
\int_{\Omega} u_{\mu}^{r} dx \geq \underline{A}^{r} \int_{S_{\varepsilon}} \left( 2\phi_{\Omega} \right)^{-\frac{2r}{p-1}} dx \geq C_{r} \int_{S_{\varepsilon}} d(x,\partial\Omega)^{-\frac{2r}{p-1}}dx > K,  
\end{align*}
and hence, $\int_{\Omega} u_{\infty}^{r} dx = \infty$. Third, we consider the case $r\leq0$. From Proposition \ref{prop:curve}, $u_{\mu}^{r}$ is decreasing if $r<0$, and thus, $u_{\mu}^{r}\leq 1$ in $\overline{\Omega}$. Additionally, from \eqref{Aline}, $u_{\mu}^{r}\geq \overline{A}^{\, r}\phi_{\Omega}^{\frac{2(-r)}{p-1}}$ in $\overline{\Omega}$. Hence, the Lebesgue dominated convergence theorem yields that $\int_{\Omega} u_{\mu}^{r}dx \rightarrow \int_{\Omega} u_{\infty}^{r} dx \in (0, \infty)$. 

Next, we verify assertion (iii). Let $r\geq1$. Then, we deduce from \eqref{boundonbdry} that $\| u_{\mu} \|_{L^{r}(\partial\Omega)} \rightarrow \infty$ as $\mu \to \infty$, implying $\| u_{\mu}\|_{W^{1,r}(\Omega)} \rightarrow \infty$ by the trace theorem \cite[Theorem 1, Section 5.5]{Ev10}. If $r<\frac{p-1}{2}$, then assertion (ii) provides that $\| \nabla u_{\mu} \|_{L^{r}(\Omega)} \rightarrow \infty$ as $\mu \to \infty$. Meanwhile, if $r\geq \frac{p-1}{2}$ and $\Omega$ is convex, then we use Poincar\'e's inequality \cite[(7.45)]{GT83} to infer that, for $S\subset \Omega$, 
\begin{align} \label{Pineq}
\| u_{\mu} - (u_{\mu})_{S} \|_{L^{r}(\Omega)}\leq C \| \nabla u_{\mu} \|_{L^{r}(\Omega)} \ \ 
\mbox{ with } \ (u_{\mu})_{S} := \frac{1}{|S|}\int_{S}u_{\mu} \, dx. 
\end{align}
From \eqref{boundonS}, for $\varepsilon > 0$ fixed, we infer that 
$\| (u_{\mu})_{S_{\varepsilon}} \|_{L^{r}(\Omega)} \leq C$ for $\mu\geq \underline{\mu}$, thus deducing from \eqref{Pineq} that 
\[
\| u_{\mu} \|_{L^{r}(\Omega)} \leq C \left( 1 + \| \nabla u_{\mu} \|_{L^{r}(\Omega)} \right). 
\]
Therefore, the desired assertion follows from assertion (ii), which completes the proof of Theorem \ref{mr}. 
\end{proof}

\end{document}